\documentclass[12pt, reqno]{amsart}
\usepackage{amssymb}
\usepackage{amsrefs}
\usepackage{amsmath}
\usepackage{color}
 
\newcommand{\tq}{\, : \,} 
\newcommand{\NN}{\ensuremath{\mathbb{N}}}
\newcommand{\RR}{\ensuremath{\mathbb{R}}}
\newcommand{\BB}{\ensuremath{\mathcal{B}}}

\newcommand{\EE}{\ensuremath{\mathcal{E}}}
\newcommand{\ee}{\ensuremath{\mathbf{e}}}
\newcommand{\xx}{\ensuremath{\mathbf{x}}}
\newcommand{\yy}{\ensuremath{\mathbf{y}}}
\newcommand{\zz}{\ensuremath{\mathbf{z}}}
\newcommand{\XX}{\ensuremath{\mathbb{X}}}
\newcommand{\YY}{\ensuremath{\mathbb{Y}}}

\newcommand{\ww}{\ensuremath{\textit{\textit{w}}}}

\newcommand{\vv}{\ensuremath{\textit{\textit{v}}}}

\newcommand{\sss}{\ensuremath{\textit{s}}}

\newcommand{\Id}{\ensuremath{\mathrm{Id}}}

\newcommand{\supp}{\mathop{\mathrm{supp}}\nolimits}
 
  \newcommand{\wMO}{\ell}
   \newcommand{\MO}{\mathit{h}}

\newtheorem{theorem}{Theorem}[section]
\newtheorem{proposition}[theorem]{Proposition}
\newtheorem{corollary}[theorem]{Corollary}
\newtheorem{lemma}[theorem]{Lemma}
\newtheorem{problem}[theorem]{Problem}

\theoremstyle{definition}

 \newtheorem{definition}[theorem]{Definition}

\theoremstyle{remark}
\newtheorem{remark}[theorem]{Remark}
 \numberwithin{equation}{section}
 \allowdisplaybreaks

\begin{document}

\title[Banach spaces with a unique greedy basis]{Existence and uniqueness of greedy bases in Banach spaces }

\thanks{}

\author[F. Albiac]{F. Albiac}\address{Mathematics Department\\ Universidad P\'ublica de Navarra\\
 Pamplona 31006\\  Spain} \email{fernando.albiac@unavarra.es}
 \author[J. L. Ansorena]{J. L. Ansorena}\address{Department of Mathematics and Computer Sciences\\
Universidad de La Rioja\\
 Logro\~no 26004\\
Spain} \email{joseluis.ansorena@unirioja.es}
 
\author[S. J. Dilworth]{S. J. Dilworth}\address{Department of Mathematics\\ University of South Carolina \\ Columbia
SC 29208 \\ USA} \email{dilworth@math.sc.edu}
 
 \author[ Denka Kutzarova]{Denka Kutzarova}
 \address{Institute of Mathematics Bulgarian Academy of Sciences\\
  Sofia\\
 Bulgaria\\  
 \textit{currentadress}
  \\  
 Department of Mathematics University of Illinois at Urbana-Champaign  \\ 
 Urbana, IL 61801
 \\
USA
  } \email{denka@math.uiuc.edu}

\keywords{greedy basis, unconditional basis, symmetric basis, Marcinkiewicz spaces,  Nakano spaces, Orlicz sequence spaces, uniqueness greedy (respectively, unconditional or symmetric) basis}

\subjclass[2000]{46B15,46B45}

\begin{abstract} Our aim is to investigate the properties of existence and uniqueness of greedy bases in  Banach spaces.  We show the non-existence of greedy basis in some Nakano spaces and Orlicz sequence spaces and produce the first-known examples of non-trivial spaces (i.e., different from $c_{0}$, $\ell_{1}$, and $\ell_{2}$)     with a unique greedy basis. 
\end{abstract}

\maketitle

\section{Introduction and background}\label{introd} 
\noindent  Let $\XX$ be a separable (real) Banach space. One of the most important problems in the isomorphic theory dating back to Banach's school  is the study of  the existence and uniqueness of  Schauder bases for $\XX$. The question
of uniqueness is formulated in a meaningful way  through the notion of equivalence of bases. Recall that two normalized (or semi-normalized) bases $(\ee_n)_{n=1}^{\infty}$ and $(\xx_n)_{n=1}^{\infty}$ of $\XX$ are called \textit{equivalent} provided a series $\sum_{n=1}^{\infty}a_{n}\ee_{n}$ converges if and only if $\sum_{n=1}^{\infty}a_{n}\xx_{n}$ converges. This is the case if and only if the map
\[
 \sum_{n=1}^{\infty}a_{n}\ee_{n}\to \sum_{n=1}^{\infty}a_{n}\xx_{n}
\]
defines an automorphism of $\XX$. That is, there exists a constant $C\ge 1$ such that
\[
\frac1C\left\Vert \sum_{n=1}^{N}a_{n}\ee_{n}\right\Vert\le \left\Vert \sum_{n=1}^{N}a_{n}\xx_{n}\right\Vert \le C\left\Vert \sum_{n=1}^{N}a_{n}\ee_{n}\right\Vert,
\]
for $a_{1},\dots,a_{N}\in \mathbb R$ and $N=1,2,\dots$.
As it happens, in every infinite-dimensional Banach space with a basis there are uncountably many non-equivalent normalized bases \cite{PelcynskiSinger1964}. Thus in order to get a more accurate  structural information on a given space using bases as a tool, one needs to restrict the discussion on their existence and uniqueness  to bases with certain special properties. 

The most useful and extensively studied  class of special bases is that of \textit{unconditional bases.}  A basis  $(\ee_n)_{n=1}^{\infty}$  of $X$ is  {\it unconditional} if $(\ee_\pi(n))_{n=1}^{\infty}$ is a basis of $X$ for any permutation $\pi$ of the indices. If a Banach space has a unique normalized unconditional basis it  has to be equivalent to all its permutations, i.e., it has to be {\it symmetric}.

For  a wide class of sequence spaces the canonical unit vector basis is the unique symmetric basis. This class contains all the Orlicz sequence spaces $\ell_{F}$ for which $\lim_{t\to 0} tF^{\prime}(t)/F(t)$ exists (\cite{LindTzaf1971a}), and also the Lorentz sequence spaces $d(w,p)$ where $p\ge 1$ and $w= (w_{n})_{n=1}^{\infty}\in c_{0}\setminus \ell_{1}$ is a nonincreasing sequence of positive numbers with $w_{1} =1$ (see \cite{ACL1973}). In particular   the $\ell_{p}$ spaces for $1\le p<\infty$ have a unique symmetric basis.
However, a complete classification of Banach spaces with a unique symmetric basis seems far from being achieved.

For a Banach space with a symmetric basis it is rather unusual to have a unique unconditional basis.  It is well-known that $\ell_{2}$ has a unique unconditional basis \cite{KotheToeplitz1934}, and a classic result of Lindenstrauss and Pe{\l }czy{\'n}ski \cite{LindenstraussPel1968} asserts that  $\ell_{1}$ and $c_{0}$ also have a unique unconditional basis.  Lindenstrauss and Zippin \cite{LindenstraussZippin1969} completed the picture by showing that those three spaces are the all and only Banach spaces in which all unconditional bases are equivalent.  

Once we have determined that a Banach space does not have a symmetric 
basis (a task that can be far from trivial) we must rethink the problem of uniqueness of unconditional basis. In fact,
an  unconditional non-symmetric  basis  admits a continuum of nonequivalent permutations (cf.\ \cite{Hennefeld1973}*{Theorem 2.1}).
Hence for Banach spaces without symmetric bases it is more natural to consider instead the question of uniqueness of unconditional bases up to (equivalence and) permutation, (UTAP) for short. Two unconditional bases $(\ee_n)_{n=1}^{\infty}$ and $(\xx_n)_{n=1}^{\infty}$ of a Banach space $X$ are said to be {\it permutatively equivalent} if there is a permutation $\pi$ of $\mathbb N$ so that $(\ee_\pi(n))_{n=1}^{\infty}$ and $(\xx_n)_{n=1}^{\infty}$ are equivalent. Then we say that $\XX$ has a (UTAP) unconditional basis $(\ee_n)_{n=1}^{\infty}$ if every unconditional basis in $\XX$ is permutatively equivalent to $(\ee_n)_{n=1}^{\infty}$. 

The first movers in this direction were Edelstein and   P. Wojtaszczyk, who   proved that finite direct sums of $c_0$, $\ell_1$ and $\ell_2$ have  a (UTAP)   unconditional basis \cite{EdelWoj1973}.
Bourgain et al.\ embarked on a comprehensive study aimed at classifying those Banach spaces with unique unconditional basis up to permutation,
 that culminated in 1985 with their {\it Memoir}  \cite{BCLT1985}. They showed that  the spaces $c_{0}(\ell_{1})$, $c_{0}(\ell_{2})$, $\ell_{1}(c_{0})$, $\ell_{1}(\ell_{2})$ and their complemented subspaces with unconditional basis all have  a (UTAP)  unconditional basis, while $\ell_{2}(\ell_{1})$ and $\ell_{2}(c_{0})$ do not. However, the hopes of attaining a satisfactory classification were shattered when they  found a nonclassical Banach space, namely the $2$-convexification  $\mathcal T^{(2)}$ of Tsirelson's space  having a (UTAP) unconditional basis. Their work also  left many open questions, most of which remain unsolved as of today. Using completely different techniques, Casazza and Kalton solved some of these problems more recently  in \cite{CasazzaKalton1998} by showing that the original Tsirelson space $\mathcal T$, and  certain complemented subspaces of Orlicz sequence spaces have a (UTAP) unconditional basis. They also found a space with a (UTAP) unconditional basis with  a complemented subspace failing to have  a (UTAP) unconditional basis.

At the end of the 20th century, Banach space theorists began to feel attracted to study a brand new type of bases  called greedy bases originating from nonlinear approximation and compressed sensing. Let us recall the definition.
For $N=1,2,\dots,$ we consider non-linear operators $G_{N}$ defined by
\[
x=\sum_{n=1}^{\infty}\ee_n^{\ast}(x)\ee_n\in X\mapsto G_N(x)=\sum_{n\in \Lambda_{N}(x)}\ee_n^{\ast}(x)\ee_n,
\]
where $\Lambda_{N}(x)$ is {\it any} $N$-element set of indices such that 
\[
\min\{|\ee_n^{\ast}(x)|: n \in \Lambda_{N}(x)\}\ge \max\{|\ee_n^{\ast}(x)|: n\not\in \Lambda_{N}(x)\}.
\]
The basis $(\ee_n)_{n=1}^{\infty}$ is said to be {\it greedy}  \cite{KoTe} if $G_N(x)$ is essentially the best $N$-term  approximation to $x$ using the basis vectors, i.e., there is a constant $C$ such that for all $x\in X$ and $N\in \mathbb N$, we have
\[
\Vert x-G_{N}(x)\Vert\le C\inf\left\{\left\Vert x-\sum_{n\in A}\alpha_n\ee_n\right\Vert : |A|=N, \alpha_n\in \RR  \right\}.
\]
Konyagin and Temlyakov  showed in \cite{KoTe} that greedy bases can be simply characterized as unconditional bases with the additional property of being {\it democratic}, i.e.,  for some constant  $\Delta>0$ we have
\[
\left\Vert \sum_{n\in A}\ee_n\right\Vert\le \Delta \left\Vert \sum_{n\in B}\ee_n\right\Vert,
\]
whenever $A$ and $B$ are finite subsets of integers of the same cardinality. Symmetric bases are trivially greedy but there exist greedy bases, such as the normalized Haar system in $L_{p}[0,1]$ for $1<p<\infty$, which are greedy but not symmetric. Thus being greedy is an intermediate property between being unconditional and being symmetric.

In this article we are concerned with investigating the novel property of  uniqueness of greedy bases  in Banach spaces in the expected sense. Since it is the first time this property is explicitly formulated in the literature we include its definition.

\begin{definition} Suppose $X$ is a Banach space with a semi-normalized greedy bases $(\ee_n)_{n=1}^{\infty}$. The space $X$ is said to have a unique greedy basis (respectively, up to permutation) if whenever  $(\xx_n)_{n=1}^{\infty}$ is another semi-normalized greedy basis of $X$, then $(\xx_n)_{n=1}^{\infty}$ is equivalent to   $(\ee_n)_{n=1}^{\infty}$ (respectively,  to a permutation of  $(\ee_n)_{n=1}^{\infty}$.
\end{definition}

Let us examine some possible scenarios. Firstly,  Banach spaces  that possess symmetric basis and have a unique unconditional basis (i.e., $\ell_{1}$, $\ell_{2}$, and $c_{0}$) also have a unique greedy basis, while if a Banach space with symmetric basis does not have a unique  symmetric basis it cannot have a unique greedy basis.
In between these two extremes we  come across  spaces like $\ell_p$ for $1<p<\infty$, $p\not=2$, with a unique symmetric basis and a wealth of  permutatively non-equivalent greedy bases (cf.\ \cite{DHK2006}).
The following question  naturally arises:
\begin{problem}\label{problemon} Are there any other Banach spaces aside from $c_{0}, \ell_{1},$ and $\ell_{2}$ with the property of uniqueness of greedy basis?
\end{problem}

Secondly, in a Banach space without a symmetric basis it may happen that  the natural basis of $X$ fails to be greedy. Such is the case in  Bessov-type spaces like $(\oplus_{n=1}^\infty \ell_q^n)_p$ and $\ell_{p}(\ell_{q})$ for $1\le p\not=q<\infty$. This makes relevant to investigate the question of existence of greedy bases.  Dilworth et al.\ \cite{DFOS2011} settled the problem of  existence of greedy bases in  $(\oplus_{n=1}^\infty \ell_q^n)_p$ by proving that these spaces have a greedy basis when 
$1<p<\infty$ and do not otherwise. In turn, Shechtmann \cite{Schectmann2014} showed that $\ell_{p}(\ell_q)$ fails to have a greedy basis in the nontrivial cases. As far as uniqueness is concerned, the right question to ask  in Banach spaces without a symmetric basis is the uniqueness of greedy basis up to a permutation.
For instance,  Smela  proved that the $L_{p}$-spaces for $1<p<\infty$, $p\not=2$, and $H_{1}$  have infinitely many permutatively non-equivalent greedy bases \cite{Smela2007} (cf. \cite{DHK2006}). 

In the $L_{p}$-spaces ($1<p<\infty, p\not=2$) we find other interesting features that are worth it singling out.  They have  greedy basis  (the normalized Haar system, \cite{Temlyakov1998}) and contain complemented subspaces such as $\ell_{p}(\ell_{2})$ and $\ell_{p}\oplus\ell_{2}$ with no greedy basis. One could argue that a reason for this pathology is that $L_{p}$ has no symmetric basis. But Pe{\l}czy{\'n}ski's universal space  denoted by $U$ has a symmetric basis and clearly contains complemented subspaces with no greedy basis (see also Theorems~\ref{ComplementedNotGreedy} and \ref{ComplementedNotGreedySR} below for more natural examples of Orlicz sequence spaces exhibiting this phenomenon).

Notice that  in order to establish the existence and/or the uniqueness up to permutation of greedy basis for
 Banach spaces  without a symmetric basis that have  (UTAP) unconditional basis,  the point is  to determine whether the unique unconditional basis is greedy. 
 This is the case with the aforementioned Tsirelson-type space $\mathcal T^{(2)}$, not to be confused with its close relative, the original Tsirelson space $\mathcal T$.
The former provides an example of a Banach space without a symmetric basis but with a (UTAP) greedy basis.The latter, despite the fact that it has a (UTAP) unconditional basis, fails to contain any democratic basic sequences (cf.\ \cite{DKK2003}*{Remark 5.8}). This observation answers in the negative what we could call {\it the greedy basic sequence problem} (evocative of the unconditional basic sequence problem).

Let us next briefly describe the contents of this article.
For expositional ease, in the preparatory Sect.~\ref{Prelim}  we have gathered some prerequisites on Musielak-Orlicz spaces and  have included the not so well-known concepts of  right/left dominant spaces and of sufficiently lattice Euclidean  spaces. In Sect.\ \ref{Nakano} we study the existence of greedy basis in complemented subspaces of Nakano spaces. In particular we see that certain Nakano spaces fail to have a greedy basis. 
 
 In Sect.~\ref{Orlicz} we turn to Orlicz sequence spaces with an eye to investigating  the existence and uniqueness of greedy bases and we settle Problem~\ref{problemon}. 
 
One could argue that  the spaces with a unique greedy basis in Sect.~\ref{Orlicz} were rigged to be close to   $\ell_{1}$ and that this facilitates their having a unique greedy basis. Perhaps. We accepted the challenge  and in Sections~ \ref{superrefExam} and~\ref{Lorentz} we find other non-trivial examples of spaces with unique greedy basis, this time far from $\ell_{1}$. The example provided in Sect.~\ref{superrefExam} is even super-reflexive. In Sect.~\ref{Lorentz} we show that the separable parts of the weak-$\ell_p$ sequence spaces
 (in contrast to the result for the $\ell_p$ spaces  mentioned above) have a unique greedy basis. 

It should be pointed out, though, that the spaces with a unique greedy basis obtained in Sections ~\ref{Orlicz}, \ref{superrefExam} and \ref{Lorentz} have a symmetric basis. Thus, the main unanswered question in this line of  research seems to be: \begin{problem}\label{problemon22} Does there exist a space with a non-symmetric unique greedy basis up to a
permutation but not a unique unconditional basis up to a permutation?
\end{problem}

In what follows  we employ the standard notation and terminology commonly used in Banach space theory (see, e.g., \cite{AlbiacKalton2006}). A {\it basis} always means a Schauder basis, and all bases will be assumed to be semi-normalized, i.e., the norm of their elements is uniformly bounded above and below.  
Given families of positive real numbers $(\alpha_i)_{i\in I}$ and $(\beta_i)_{i\in I}$, the notation $\alpha_i\lesssim \beta_i$ for all $i\in I$ means that 
$\sup_{i\in I}\alpha_i/\beta_i <\infty$, while $\alpha_i\approx \beta_i$ for all $i\in I$ means that  $\alpha_i\lesssim \beta_i$ and  $\beta_i\lesssim \alpha_i$ for all $i\in I$.
Applied to Banach spaces $\XX$ and $\YY$, the notation $\XX \approx \YY$ means that $\XX$ and $\YY$ are isomorphic.

To quantify the democracy of a basis $\BB=(\ee_n)_{n=1}^\infty$ we will consider  the {\it upper democracy function} (also called the {\it fundamental function}) of $\BB$ in $\XX$, given by
 \[
 \varphi_{u}[\BB,\XX](N)=\sup_{|A|\le  N}\left\Vert \sum_{n\in A}\xx_n\right\Vert, \qquad N=1,2,\dots,
 \]
and the {\it lower democracy function} of $\BB$ in $\XX$, 
 \[
 \varphi_{l}[\BB,\XX](N)=\inf_{|A|\ge N}\left\Vert \sum_{n\in A}\xx_n\right\Vert,\qquad N=1,2,\dots
 \]
A basis $\BB$ is democratic if and only if $\varphi_u[\BB,\XX](N)\approx \varphi_l[\BB,\XX](N)$ for all $N\in\NN$.
Any other more specific notation will be introduced in context when needed.

 \section{Preliminaries}\label{Prelim}
 
 \noindent We summarize some basic facts about Musielak-Orlicz sequence spaces \cite{Musielak1983}  adapted to our needs.

A (normalized) \textit{convex Orlicz function} is a  nonnnegative and nonzero convex function $F\colon[0,\infty)\to[0,\infty)$ such that $F(0)=0$ and $F(1)=1$.

Let $\omega$ be the set of real-valued sequences.  

Given a  (normalized) \textit{convex Musielak-Orlicz sequence} $(F_n)_{n=1}^\infty$, i.e.,   $(F_n)_{n=1}^\infty$ are  convex Orlicz functions, we define the
\textit{Musielak-Orlicz modular} $m_{(F_n)}\colon \omega  \to [0,\infty]$ by
\begin{equation}\label{modularMuseilak}
(a_n)_{n=1}^\infty \mapsto m_{(F_n)}((a_n)_{n=1}^\infty))=\sum_{n=1}^\infty  F_n(|a_n|).
\end{equation}
The  \textit{Musielak-Orlicz norm} $\Vert \cdot \Vert_{(F_n)}\colon  \omega \to [0,\infty]$ is then given by
\begin{equation}\label{normMuseilak} \alpha \mapsto \inf\{t\in(0,\infty)  \tq m_{(F_n)}(\alpha/t) \le 1\},
\end{equation}
and its asociated  (full) \textit{Musielak-Orlicz space} is $(\wMO_{(F_n)},\Vert \cdot \Vert_{(F_n)})$, where
\begin{equation}\label{fullMuseilak} 
\wMO_{(F_n)} =\{ \alpha\in\omega \tq m_{(F_n)}(\alpha/t)<\infty \text{ for some } t>0\}.
\end{equation}
The space $(     \wMO_{(F_n)},\Vert \cdot \Vert_{(F_n)}   )$ is a Banach space 
and the canonical unit vectors  
$(\ee_n)_{n=1}^\infty$ 
form a $1$-unconditonal basic sequence in $ \wMO_{(F_n)}.$  

The (restricted) Musielak-Orlicz sequence space, denoted by $\MO_{(F_n)}$, is  the closed linear span of $(\ee_n)_{n=1}^\infty$ in $\wMO_{(F_n)}$. We have
\begin{equation}\label{restrictedMuseilak} 
\MO_{(F_n)}=\{ \alpha\in\omega \tq m_{(F_n)}(\alpha/t)<\infty \text{ for all } t>0\}.
\end{equation}
Musielak proved the following result:
\begin{theorem}[\cite{Musielak1983}*{Theorem 8.14}]\label{MusielakDensity} Let $=(F_n)_{n=1}^\infty$ be a convex Musielak-Orlicz sequence.  Then $\wMO_{(F_n)}=\MO_{(F_n)}$ if and only  if there exist a positive sequence $(a_n)_{n=1}^\infty \in \ell_1$, some
 $0<\delta<\infty$, and a constant $1<C<\infty$  such that
 \[
 F_n(t) < \delta \Rightarrow F_n(2 t) \le C F_n(t) + a_n.
 \]
\end{theorem}
A similar theorem  characterizes the inclusions (hence, the identifications) between different  Musielak-Orlicz sequence  spaces. 
\begin{theorem}[\cite{Musielak1983}*{Theorem 8.11}] \label{MusielakInclusion} Let $(F_n)_{n=1}^\infty$ and $(G_n)_{n=1}^\infty$  be two  convex Musielak-Orlicz sequences. Then 
$\wMO_{(F_n)}\subseteq \wMO_{(G_n)}$ if and only if there exist a positive sequence $(a_n)_{n=1}^\infty \in \ell_1$, some $\delta >0$,  and positive constants  $b$ and $C$ such that
\[
G_n(t) < \delta \Rightarrow F_n(t) \le C G_n(b t) + a_n.
\]
\end{theorem}
Note that a Musielak-Orlicz norm is determined by its values on $c_{00}$. Therefore, $\wMO_{(F_n)}\subseteq \wMO_{(G_n)}$ if and only if $\MO_{(F_n)}\subseteq \MO_{(G_n)}$.

Given $A\subseteq \NN$ we denote by $\MO_{(F_n)}[A]$ (respectively,  $\wMO_{(F_n)}[A]$) the complemented subspace of  $\MO_{(F_n)}$ (respectively, complemented of  $\wMO_{(F_n)}$) consisting of all sequences of the space supported in $A$. If $\pi\colon\NN\to \NN$ is an injective mapping and $A=\pi(\NN)$, the sequence
$(\ee_{\pi(n)})_{n=1}^\infty$ is a basis of  $\MO_{(F_n)}[A]$ isometrically equivalent to the unit vector basis of the Musielak-Orlicz space $\MO_{(F_{\pi(n)})}$.
This elementary and useful property can be  formulated in terms of direct sums: given two   convex Musielak-Orlicz sequences $(F_n)_{n=1}^\infty$ and 
$(G_n)_{n=1}^\infty$ the unit vector basis of $\MO_{(F_n)}\oplus \MO_{(G_n)}$ is equivalent to the unit vector basis of $\MO_{(H_n)}$, where 
\[
(H_n)_{n=1}^\infty=(F_1,G_1,F_2,G_2,\dots, H_{2n-1},H_{2n},\dots)
\]
This, combined  with the asymmetry in the canonical basis caused by the variation in $n$ of the Orlicz function $F_n$,  makes of Musielak-Orlicz spaces a suitable place to search for bases which are not equivalent to their square (cf.\ \cite{CasazzaKalton1998}*{Proposition 6.8}).

A  Banach sequence space $\XX$ is called \textit{right dominant} if there is  a constant $C$ such that whenever $(x_j)_{j=1}^N$ and  $(y_j)_{j=1}^N$ in $\XX$ are such that 
$\supp x_i\cap\supp x_j=\supp y_i\cap\supp y_j=\emptyset$ for $i\not=j$ and, for $j=1,\dots, N$ we have
$ \Vert x_j \Vert  \le \Vert y_j \Vert$ and $\max\{n \colon n\in\supp x_j \} < \min\{n \colon n\in\supp y_j \}$
 it implies
 \[
 \left\Vert \sum_{j=1}^N x_j \right\Vert \le C  \left\Vert \sum_{j=1}^N y_j \right\Vert.
 \]
 
 A Banach space $\XX$ is called \textit{sufficiently (lattice) Euclidean} if there is a constant $\lambda$ such that for any $n$ there are operators  $S_n\colon X\to \ell_2^n$  and $T_n\colon \ell_2^n \to X$ such that  $S_n$ and $T_n$ are a (lattice) homomorphism, $S_n\circ T_n =\Id_{\ell_2^n}$, and 
$\Vert S_n \Vert  \Vert T_n \Vert \le \lambda$. We will say that $\XX$ is \textit{(lattice) anti-Euclidean} if it is not sufficiently (lattice) Euclidean.
An unconditional basic sequence $\BB =(\xx_{n})_{n=1}^{\infty}$ induces a lattice structure on its closed linear span $\YY=[\xx_{n}\colon n\in \NN]$. We say that  $\BB$ is 
lattice anti-Euclidean if $\YY$ is  lattice anti-Euclidean.

\section{Existence of greedy bases in Nakano spaces}\label{Nakano}

\noindent If $1\le p<\infty$ the map $t\mapsto t^p$ is a  convex Orlicz function.  Given  a sequence  $(p_n)_{n=1}^\infty$ in $[1,\infty)$,  the Musielak-Orlicz spaces defined as in \eqref{fullMuseilak} and \eqref{restrictedMuseilak} for the sequence of Orlicz functions $(F_n)_{n=1}^\infty$, where
\[
F_n(t)=t^{p_n}
\]
are denoted them by $\wMO_{(p_n)}$ and $\MO_{(p_n)}$ and are called {\it Nakano spaces}.
The corresponding modular and norm, as defined  in \eqref{modularMuseilak} and \eqref{normMuseilak}, for this particular case   will be denoted respectively by  $m_{(p_n)}$ and  $\Vert \cdot \Vert_{(p_n)}$.

These  spaces were introduced by Nakano in 1950 \cite{Nakano1950}. In  \cite{Nakano1951}, he completely  characterized the inclusions  between these spaces. In 1965, Simmons \cite{Simmons1965} obtained a similar result in the non-locally convex setting (i.e., when $p_n\le 1$ for all $n$). We refer to \cite{BlascoGregori2001} for a proof  of Theorem~\ref{NakanoInclusion} based on Theorem~\ref{MusielakInclusion}.
\begin{theorem}\label{NakanoInclusion}
Let $(p_n)_{n=1}^\infty$  and $(q_n)_{n=1}^\infty$ be  two sequences in $[1,\infty)$. Then  $\wMO_{(p_n)} \subseteq \wMO_{(q_n)} $ (with continuous inclusion) if and only if
there is $0<r<1$ such that
 \[
 \sum_{q_n<p_n} r^{ p_n q_n / (p_n -q_n)} <\infty.
\]
\end{theorem}

\begin{corollary}\label{NakanoIdentification}
Let $(p_n)_{n=1}^\infty$  and $(q_n)_{n=1}^\infty$ be two Nakano indices. Then  $\wMO_{(p_n)} = \wMO_{(q_n)} $  (with equivalent norms)
if and only if there is $0<r<1$ such that
 \[
 \sum_{n=1}^\infty r^{ p_n q_n / |p_n -q_n|} <\infty.
\]
\end{corollary}
Notice that inclusions between Nakano spaces work as in the $\ell_p$ spaces, in the sense that  if $(p_n)_{n=1}^\infty$ and $(q_n)_{n=1}^\infty$ are Nakano indices  satisfying  $p_n\le q_n$ for all $n\in\NN$, then $\wMO_{(p_n)} \subseteq\wMO_{(q_n)} $.  We will need a more quantitative formulation of this result valid even for \textit{finite dimensional Nakano spaces}.
\begin{lemma}\label{NormOneInequality}Let $\eta\in\NN\cup\{\infty\}$. Let $(p_n)_{n=1}^\eta$, $(q_n)_{n=1}^\eta$ such that $1\le p_n\le q_n <\infty$.  Then
\[
\left\Vert  (a_n)_{n=1}^\eta \right\Vert_{(q_n)} \le \left\Vert  (a_n)_{n=1}^\eta \right\Vert_{(p_n)} 
\]
for all $(a_n)_{n=1}^\eta$ in $\RR$.
\end{lemma}

The following theorem is crucial  in Casazza-Kalton's discussion  on Nakano spaces in \cite{CasazzaKalton1998}.
\begin{theorem}\label{RightDom}
Let $(p_n)_{n=1}^\infty$ be  a decreasing sequence in $[1,\infty)$. Then  $\MO_{(p_n)}$
is a right-dominant sequence space.
 \end{theorem}

The  spaces $\ell_p$ for $1\le p<\infty$ are  Nakano spaces $\wMO_{(p_n)}$ in which  $(p_{n})$ is the constant  sequence $p$, while the space
 $\ell_\infty$ is not,  a priori, a  Nakano space. However, we can state a theorem similar to Corollary~\ref{NakanoIdentification} that  characterizes when $\wMO_{(p_n)}$ coincides with the space $\ell_\infty$ (i.e., the corresponding Nakano space $\MO_{(p_n)}$ agrees with $c_0$).  
 The proof is similar to that of Theorem~\ref{NakanoInclusion}.
\begin{theorem}\label{NakanoInftyIdentification}
Let $(p_n)_{n=1}^\infty$  be a sequence in $[1,\infty)$. Then  $\wMO_{(p_n)}=\ell_\infty$ (with equivalent norms) if and only if 
there is $0<r<1$ such that
 $ \sum_{n=1}^\infty r^{p_n} <\infty$.
 \end{theorem}

Let us next enunciate a result that  characterizes when $\wMO_{(p_n)}$ and $\MO_{(p_n)}$ agree. 
\begin{theorem}[\cite{Nakano1951}]\label{NakanoDensity} Let $(p_n)_{n=1}^\infty$ be a sequence in $[1,\infty)$.
Then $\wMO_{(p_n)}=\MO_{(p_n)}$ if and only if $\sup_n p_n<\infty$.
\end{theorem}

The same condition as in Theorem~\ref{NakanoDensity} allows us to characterize  boundedness in Nakano spaces without appealing to the Nakano norm.
\begin{proposition}\label{NakanoBoundedness} Suppose  $(p_n)_{n=1}^\infty$ is a sequence in $[1,\infty)$ such that $\sup_n p_n<\infty$. Let  $\mathcal{A}$ be  a subset in $\wMO_{(p_n)}$.
Then: 
\begin{enumerate} 

\item[(a)] $\mathcal{A}$  is 
norm-bounded above (i.e., $\sup_{\alpha\in \mathcal{A}} \Vert \alpha\Vert_{(p_n)}<\infty$) if and only if it is modular-bounded  above (i.e., $\sup_{\alpha\in \mathcal{A}} m_{(p_n)}(\alpha)<\infty$).

\item[(b)]   $\mathcal{A}$  is 
norm-bounded below (i.e., $\inf_{\alpha\in \mathcal{A}} \Vert \alpha\Vert_{(p_n)}>0$)  if and only if it is modular-bounded below
( i.e., $\inf_{\alpha\in \mathcal{A}} m_{(p_n)}(\alpha)>0$).
\end{enumerate}
\end{proposition}

\begin{proof} Let $s=\sup_n p_n$. The result is  an easy consequence of the estimates
\[
\min\{ \Vert \alpha \Vert_{(p_n)} , \Vert \alpha \Vert_{(p_n)}^s \} \le m_{(p_n)}(\alpha) \le \max\{ \Vert \alpha \Vert_{(p_n)} , 
\Vert \alpha \Vert_{(p_n)}^s \}
\]
for any $\alpha\in\omega$. \end{proof}
We would like to point out the close connection   between Proposition~\ref{NakanoBoundedness} and the equivalence between the \textit{norm convergence} and the \textit{modular convergence} obtained in \cite{Musielak1983}*{Theorem 8.14}.  However,  Proposition~\ref{NakanoBoundedness} provides   a formulation  more  fit for our purposes.

Duality in Nakano space works as expected. We refer the reader to \cite{Musielak1983}*{Theorem 13.11} for a more general result in the setting of Musielak-Orlicz spaces.
\begin{theorem}[\cite{Nakano 1950}]\label{NakanoDuality} Let $(p_n)_{n=1}^\infty$ be a sequence in $[1,\infty)$. Let $(q_n)_{n=1}^\infty$ be  a sequence in $(1,\infty)$ such that $1/p_n+1/q_n=1$ if $p_n>1$ and, for some $0<r<1$,  $\sum_{p_n=1} r^{q_n}<\infty$. 
Then $\MO_{(p_n)}^*=\wMO_{(q_n)}$ with the natural duality pair and equivalent norms.
\end{theorem}

We are now in a position to state and prove our first results about bases in Nakano spaces.

\begin{lemma}~\label{clusterPointDemocracyEstimate}Let $(p_n)_{n=1}^\infty$ be a sequence in $[1,\infty)$ and let $p$ be  a cluster point of $(p_n)_{n=1}^\infty$. Denote by $\EE$ the unit vector basis. Then, for all $N\in\NN$,
\[
\varphi_l[\EE,\MO_{(p_n)}](N)\le  N^{1/p} \le \varphi_u[\EE,\MO_{(p_n)}](N).
\]
\end{lemma}
\begin{proof}Fix $N\in\NN$.
Let $r<p<s$. There is $A\subseteq\NN$ such that $|A|=N$ and $r<p_n<s$ for all $n\in A$. By Lemma~\ref{NormOneInequality},
\[
 N^{1/s}=  \left\Vert \sum_{n\in A} \ee_n\right\Vert_{s}  \le \left\Vert \sum_{n\in A} \ee_n\right\Vert_{(p_n)}  \le  \left\Vert \sum_{n\in A} \ee_n\right\Vert_{r}= N^{1/r}.
 \]
 Hence,
  \[
  \varphi_l[\EE,\MO_{(p_n)}](N) \le N^{1/r}\quad \text{and}\quad   \quad N^{1/s} \le  \varphi_u[\EE,\MO_{(p_n)}](N).
 \]
 Choosing $r$ and $s$ arbitrarily close to $p$ we get the desired result.
\end{proof}

\begin{theorem}\label{uvbNakanoGreedy}Let $(p_n)_{n=1}^\infty$ be a sequence in $[1,\infty)$. Then the unit vector basis is a greedy basis for the Nakano space 
$\MO_{(p_n)}$ if and only if there is $p\in[1,\infty]$ such that
 $\MO_{(p_n)}=\ell_p$, i.e., if and only if the  unit vector basis  of $\MO_{(p_n)}$ is equivalent to the  unit vector basis   of $\ell_p$ ($c_0$ if $p=\infty$), in which case $\lim_n p_n=p$.
\end{theorem}

\begin{proof}Suppose that the unit vector basis $\EE=(\ee_n)_{n=1}^\infty$ is a greedy basis of $\MO_{(p_n)}$. Then, in particular,  $\EE$   is democratic.  

\smallskip 
\underline{Case 1:}  The sequence $(p_n)_{n=1}^\infty$ decreases to $p\in[1,\infty)$. 

Let $N\in\NN$.  By Lemma~\ref{NormOneInequality}.
 \[
 \left\Vert \sum_{n=1}^N \ee_n\right\Vert_{(p_n)}  \le  \varphi_l[\EE,\MO_{(p_n)}](N)\le \varphi_u[\EE,\MO_{(p_n)}](N) \le
 \left\Vert \sum_{n=1}^N \ee_n\right\Vert_{p}
 = N^{1/p}.
 \]
 
 Hence,  by Lemma~\ref{clusterPointDemocracyEstimate},  
 \[
 \varphi_l[\EE,\MO_{(p_n)}](N)  = \left\Vert \sum_{n=1}^N \ee_n\right\Vert_{(p_n)}\quad \text{and}\quad    \varphi_u[\EE,\MO_{(p_n)}](N)=N^{1/p}.
 \]

Since  $\EE$  is  democratic,
\[
\inf_N \frac{  \varphi_l[\EE,\MO_{(p_n)}](N)   }{   \varphi_u  [\EE,\MO_{(p_n)}](N) } =\inf_{N\in\NN} \left\{  N^{-1/p} \left \Vert \sum_{n=1}^N \ee_n \right\Vert_{(p_n)}  \tq N\in\NN \right\}>0.
\]
By Lemma~\ref{NakanoBoundedness}, there is $0<c<1$ such that
\[
 \sum_{n=1}^N  N^{-p_n/p}\ge c, \quad N=1,2,\dots
\]
Therefore, since $p_n\ge p_N$ for $n\le N$,
\[
N^{1-p_N/p}\ge c, \quad N=1,2,\dots
\]
Hence,
\[
\frac{1}{p}-\frac{1}{p_N} \le \frac{p_N}{p}-1\le \log\left(\frac{1}{c}\right) \frac{1}{\log(N)},  \quad N=1,2,\dots
\]
If $0<r<c$, 
\[
\sum_{N=1}^\infty r^{p p_N/(p_N-p)}\le\sum_{N=1}^\infty N^{-\log (r)/log(c)}<\infty.
\]
By Theorem~\ref{NakanoIdentification}, $\MO_{(p_n)}=\ell_p$.

\smallskip 
\underline{Case 2:} The sequence $(p_n)_{n=1}^\infty$ increases to $p\in[1,\infty)$.

It is similar to  Case 1 and we leave   the details  for the reader.

%Suppose that $(p_n)_{n=1}^\infty$ increases to $p\in[1,\infty)$.  We proceed similarly as in (a). We obtain
 %\[
 %\sup_N N^{-1/p}  \left\Vert \sum_{n=1}^N \ee_n \right\Vert_{(p_n)} =\sup_N \frac{  \varphi_u[\EE,\MO_{(p_n)}](N)   }{   \varphi_l  [\EE,\MO_{(p_n)}](N) }<\infty.
 %\]
% Appealing to Lemma~\ref{NakanoBoundedness}, 
 %\[
%C:=\sup_N N^{1-p_N/p} \le\sup_N   \sum_{n=1}^N N^{-p_n/p}<\infty.
 %\]
 %Therefore, if $r<1/C$,
%\[
%\sum_{N=1}^\infty r^{p p_N/(p_N-p)}\le\sum_{N=1}^\infty N^{-log(1/r)/\log(C)}<\infty.
%\]
%By Theorem~\ref{NakanoIdentification}, $\MO_{(p_n)}=\ell_p$.

\smallskip 
\underline{Case 3:}  The sequence $(p_n)_{n=1}^\infty$ converges to $p\in[1,\infty)$.  

Consider $A_1=\{n\in\NN \tq p_n\le p\}$ and $A_2=\{ n\in\NN \tq p_n>p\}$.
Denote $\NN_j=\{ n\in\NN \tq n \le |A_j|\}$ 
%and  by $\XX_j$ the closed linear span of $\{ \ee_j \tq j \in A_j\}$  
 ($j=1,2$).
There is an increasing bijection from $\NN_1$ onto  $A_1$.  
Appealing to the Case 2  we obtain $\MO_{(p_n)}[A_1]=\ell_p[A_1]$.   Similarly, there is a   decreasing bijection from  $\NN_2$ onto $A_2$. Appealing to the case (a)   we obtain $\MO_{(p_n)}[A_2]=\ell_p[A_2]$. Hence $\MO_{(p_n)}=\MO_{(p_n)}[A_1]\oplus \MO_{(p_n)}[A_2]=\ell_p[A_1]\oplus \ell_p[A_2]=\ell_p$.

\smallskip 
\underline{Case 4:}  $\lim_n p_n=\infty$. 

Combining the democracy of the $\EE$ with  Lemma~\ref{clusterPointDemocracyEstimate},  we obtain  that  $\varphi_u[\EE,\MO_{(p_n)}](N)\lesssim 1$ for all $N\in\NN$.
Taking into account that $\EE$  is an unconditional basis, we get   $\MO_{(p_n)}=c_0$.
 
  \smallskip 
\underline{Case 5:} The sequence $(p_n)_{n=1}^\infty$  has no limit. 

Denote $p_1=\liminf_n p_n<p_2=\limsup_n p_n$. Combining the democracy of $\EE$ with Lemma~\ref{clusterPointDemocracyEstimate} we get $N^{1/p_1} \lesssim N^{1/p_2}$ for $N\in\NN$, an absurdity.
\end{proof}

\begin{remark}The proof of the above theorem gives  that for a monotone sequence $(p_n)_{n=1}^\infty$ converging to $p<\infty$,   one has
$\ell_{(p_n)}=\ell_p$ if and only if
\[
\sup_{n\in\NN} \log(n) |p_n-p| <\infty.
\]
A similar result was  obtained by  Simmons \cite{Simmons1965}  in the non-locally convex setting. 
\end{remark}

\begin{theorem}\label{GreedyNakano}
Let $(p_n)_{n=1}^\infty$ be a  sequence in $[1,\infty)$ with  $\lim_n p_n =1$.
%and there is $0<r<1$ such that
%\[
%\sum_{n=1}^\infty r^{1/|p_n-p_{2n}|}<\infty.
%\]
\begin{itemize} 
\item[(a)] Any complemented greedy basic sequence in $\wMO_{(p_n)}$ is equivalent to the unit vector basis of $\ell_1$. 
\item[(b)] A complemented subspace of $\wMO_{(p_n)}$ has a greedy basis if and only if it is isomorphic to $\ell_1$.
\end{itemize}
\end{theorem}

\begin{proof} Part (b) follows readily  from (a).  To prove (a), let us assume without loss of generality that $p_n \searrow 1$.
 Then, by Theorem~\ref{RightDom}, $\wMO_{(p_n)}$ is right-dominant.  Theorem~\ref{NakanoIdentification}  yields that the unit vector basis of  $\wMO_{(p_n)}$ has a subsequence equivalent to the unit vector basis of $\ell_1$. In particular, $\ell_1$ is disjointly finitely representable in $\wMO_{(p_n)}$. Taking into account
  \cite{CasazzaKalton1998}*{Proposition 5.3}, we obtain that $\wMO_{(p_n)}$ is anti-Euclidean.
%By Theorem~\ref{NakanoIdentification}, the unit vector basis of $\wMO_{(p_n)}$ is equivalent to the unit vector basis of $\wMO_{(p_n)} \oplus \wMO_{(p_n)}$. 

Let $\BB$ be a complemented greedy basis in $\wMO_{(p_n)}$. Then $\BB$ is lattice anti-Euclidean.
By  \cite{CasazzaKalton1998}*{Theorem 3.5}, $\BB$ is  equivalent to a  complemented block basis of the unit vector basis of  $\wMO_{(p_n)}^M$  for some $M\in\NN$.

Pick $(q_n)_{n=1}^\infty$ such that $q_n=p_m$ whenever $(m-1)M+1\le n \le mM$. 
Notice that $\wMO_{(p_n)}^M=\wMO_{(q_n)}$. In particular,  $\wMO_{(p_n)}^M$ is a right-dominant sequence space. Appealing to 
 \cite{CasazzaKalton1998}*{Theorem 5.6} we get that $\BB$ is  permutatively equivalent to a subsequence of the unit vector basis of $\wMO_{(q_n)}$.
Therefore, there is an injective mapping $\pi\colon\NN\to \NN$ such that $\BB$ is equivalent to  the unit vector basis of the  Nakano space
$\wMO_{(q_{\pi(n)})}$.
Hence, by
Theorem~\ref{uvbNakanoGreedy}, $\BB$ is equivalent to the unit vector basis of $\ell_1$.
\end{proof}

\begin{corollary}\label{UniqueGreedyNakano}
Let $(p_n)_{n=1}^\infty$ be a sequence in $[1,\infty)$ such that  $\lim_n p_n =1$ and
\[
\sum_{n=1}^\infty r^{1/|p_n-1|}=\infty \text{ for all } 0<r<1.
\]
Then $\wMO_{(p_n)}$ does not have a greedy basis.
\end{corollary}
\begin{proof} %Although the result is straightforward combining Theorem~\ref{uvbNakanoGreedy} with \cite{CasazzaKalton1998}*{Theorem 5.8}, in order to provide a proof as self contained as possible, we avoid the use this theorem from P. Cassaza and N. Kalton.
Suppose that $\BB$ is a  greedy basis of $\wMO_{(p_n)}$. By Theorem~\ref{GreedyNakano}, $\wMO_{(p_n)} \approx \ell_1$.
Since $\ell_1$ has a unique unconditional basis (cf.\ \cite{ LindenstraussPel1968}), $\wMO_{(p_n)}=\ell_1$, in contradiction with
Theorem~\ref{NakanoIdentification}. 
\end{proof}

Next we obtain analogous results  to Theorem~\ref{GreedyNakano} and Corollary~\ref{UniqueGreedyNakano}  for the dual case, i.e., when $\lim_n p_n=\infty$. Although we could use similar techniques in their proofs, we will get more with simpler techniques.

\begin{theorem}Let $(p_n)_{n=1}^\infty$ be a  sequence in $[1,\infty)$ such that  $\lim_n p_n =\infty$.
\begin{enumerate} 
\item[(a)] Any  greedy basic sequence in $\MO_{(p_n)}$ is equivalent to the unit vector basis of $c_0$. 
\item[(b)] A subspace of $\MO_{(p_n)}$ has a greedy basis if and only if it is isomorphic to $c_0$.
\item[(c)]  The space $\MO_{(p_n)}$ has a greedy basis if and only if  $\MO_{(p_n)}=c_0$.
\end{enumerate}
\end{theorem}

\begin{proof} (a) Let $\BB$ be a greedy basis sequence in $\MO_{(p_n)}$. Since the dual space of $\MO_{(p_n)}$ is separable (see Theorem~\ref{NakanoDuality} and
 Theorem~\ref{NakanoDensity}), $\BB$ is  a weakly null sequence (see
\cite{AlbiacKalton2006}*{Proposition 3.2.7} and \cite{AlbiacKalton2006}*{Theorem 3.3.1}). 
Applying Bessaga-Pe{\l}czy{\'n}ski Selection Principle (cf. \cite{AlbiacKalton2006}*{Proposition 1.3.10}) we get a normalized block basis of the the unit vector basis of $\MO_{(p_n)}$, say  $\BB_0=(\xx_k)_{k=1}^\infty$, equivalent to  a subbasis of $\BB$. 
Denote $A_k=\supp \xx_k$ and pick $n_k\in A_k$ such that $q_k:=p_{n_k}\le p_n$ for $n\in A_k$. Passing to a subsequence, if necessary, we can suppose that
$\sum_{k=1}^\infty  r^{q_k}<\infty$ for some $0<r<1$. By Theorem~\ref{NakanoInftyIdentification}, $\MO_{(q_k)}=c_0$. Then
\[
\left\Vert \sum_{k=1}^\infty a_k \xx_k \right\Vert_{(p_n)}  \le \left\Vert \sum_{k=1}^\infty a_k \ee_k \right\Vert_{(q_k)} \approx 
 \left\Vert \sum_{k=1}^\infty a_k \ee_k \right\Vert_\infty,
 \]
 for all $(a_n)_{n=1}^\infty\in c_{00}$. Therefore,
 \[
 \varphi_u[\BB, \MO_{(p_n)}] (N)  \approx  \varphi_u[\BB_0, \MO_{(p_n)}] (N) \lesssim \varphi_u[\EE, c_0] (N)=1.
 \]
 From here, taking into account that $\BB$ is an unconditional basis, we get readily that $\BB$ is equivalent to the unit vector basis of $c_0$.

Part (b) is an easy consequence of (a) and Part (c) follows from (b) and the uniqueness of unconditional basis in $c_0$.
\end{proof}

\section{Uniqueness of greedy basis in Orlicz sequence spaces}\label{Orlicz}

\noindent Orlicz sequence spaces can be seen as a particular case of Musielak-Orlicz sequence spaces. Indeed, we just need to consider a sequence $(F_n)_{n=1}^\infty$ such that $F_n=F$ for all $n$ and some convex Orlicz function.
We put  $\wMO_{(F_{n})}=\wMO_F$ and  $\MO_{(F_{n})}=\MO_F$. The identification between
$\wMO_F$ and $\MO_F$  is simplier than  for Musielak-Orlicz sequence spaces: $\wMO_F=\MO_F$ if and only if $F$ satisfies the $\Delta_2$ condition  at the origin, i.e., there exist constants $a\in(0,\infty)$ and $C\in(1,\infty)$  such  that  $F(2t)   \le C F(t)$ for $t\in[0,a]$.
Notice that   Orlicz spaces only depend, up to an equivalent norm,  of the behavior of the functions defining them at a neighborhood of the origin. To be precise,  $\wMO_{F}=\wMO_{G}$ if and only if there exist positive constants $a$ and $b$ such that 
$F(b t)\approx G(t)$ for all $t\in[0,a]$.

In an Orlicz space $\MO_F$,  the unit vector basis is a $1$-symmetric basis. In particular, it is a greedy basis. 
Its democracy functions $\varphi_l[\EE,\MO_F](N)=\varphi_u[\EE,\MO_F](N)=D_N$ are determined by the formula
\[
F\left( \frac{1}{D_N}\right)=\frac{1}{N}, \quad N=1,2,\dots
\]
From here it is easy to obtain the following result.
%Lo hicieron Gustavo y Eugenio?
\begin{lemma}~\label{SameDemocracyFunction} Let $F$ and $G$ be two  convex Orlicz functions.  Then
 $\MO_F=\MO_G$ (with equivalence of norms) if and only if $ \varphi_u[\EE,\MO_F](N)\approx \varphi_u[\EE,\MO_G](N)$ for all $N\in \mathbb N$.
\end{lemma}

Another elementary property of interest for us is that the unit vector basis of $\MO_F\oplus \MO_F$ is equivalent to the unit vector basis of $\MO_F$.

We will need to consider Musielak-Orlicz sequence spaces arising from the flows of an Orlicz function.
To be precise,  fix a  convex Orlicz function $F$,
and for $0<s<\infty$ consider
\[
F_s\colon[0,\infty)\to[0,\infty), \quad t\mapsto \frac{F(st)}{F(s)}.
\]
Given a sequence $(s_n)_{n=1}^\infty$ in $(0,\infty)$ we define
\[
\MO_F[s_n]:=\MO_{(F_{s_n})}.
\]
The following result,  implicitly stated in \cite{CasazzaKalton1998}, establishes  the connection between this kind of Musielak-Orlicz sequence spaces and block bases in Orlicz sequence spaces.
\begin{lemma}\label{BlockOrliczEquivalentMusielak} Let $F$ be a  convex Orlicz function. 
\begin{enumerate}
\item[(i)] Let $\BB=(\xx_n)_{n=1}^\infty$ be a constant-coefficient normalized block basic sequence  of the unit vector basis of $\MO_F$. For each $n\in \mathbb N$, 
denote by $N_n$ the lenght of the block $\xx_n$ 
and  let $s_n\in(0,\infty)$  be such that  $N_n  F(s_n) =1$. Then $\BB$ is a complemented basic sequence in $\MO_F$ isometrically equivalent to the unit vector basis of 
$\MO_F[s_n]$. 

\item[(ii)]   Let $(s_n)_{n=1}^\infty$ be a bounded sequence of positive numbers. Then  the unit vector basis of $\MO_F[s_n]$ is equivalent to a constant-coefficient block basis of the unit vector basis of  $\MO_F$.
\end{enumerate}
\end{lemma}

Next we focus on  convex Orlicz functions that are equivalent at the origin to $t\mapsto t^p(-\log t)^{-a}$  for some $1\le p$ and $a>0$. To be precise, put 
 \begin{align}\label{OurOrliczFunction}
F^{p,a}(t)=\begin{cases}   
e^{-pa} t^p(-\log t)^{-a} &\text{if}\; 0<t <\frac{1}{e}, \\
t^{p+a}   &\text{if}\; \frac{1}{e}\le t <\infty.
\end{cases}
\end{align}
Let us recall some properties of these Orlicz functions. Denote $F=F^{p,a}$.

\begin{itemize} 

\item $F$ is multiplicatively convex, i.e.,
\[
F(s^\theta t^{1-\theta})\le F(s)^\theta F(t)^{1-\theta} \text{ whenever }0<s,t,\theta<1.
\]

\item $\wMO_F=\MO_F$.

\item If $p=1$, $\wMO_{F}$ is anti-Euclidean (cf. \cite{CasazzaKalton1998}*{Lemma 6.2}).

\item 
 $\wMO_{F}$ and $\ell_p$ are the unique Orlicz spaces that are  subspaces of $\wMO_{F}$  (cf. \cite{LinTza1977}*{Theorem 4.a.8}).

 \item  $\wMO_{F}$ has a unique symmetric basis (cf.  \cite{LinTza1977}*{Proposition 4.b.10}). However,  $\wMO_{F}$ does not have a unique unconditional basis  (cf. \cite{LindenstraussZippin1969}).

 \end{itemize}

\begin{remark}\label{complementedlp} Let  $F=F^{p,a}$ and suppose $\lim_n s_n=0$. It follows from \cite{CasazzaKalton1998}*{Proposition 6.6}  that
  the unit vector basis of $\MO_{F}[s_n]$ has a subbasis equivalent  to the unit vector basis of $\ell_p$. 
Combining with Lemma~\ref{BlockOrliczEquivalentMusielak} we obtain that $\ell_p$ is a complemented subspace of $\wMO_F$.
\end{remark}

\begin{proposition}\label{MusielakGreedy} Let $p\in[1,\infty)$ and $a\in(0,\infty)$. Let $(s_n)_{n=1}^\infty$ be a sequence such that $\lim_n s_n = 0$.  Consider $F=F^{p,a}$ defined as is \eqref{OurOrliczFunction}. The following are equivalent:
\begin{enumerate}
\item[(a)] $\MO_{F}[s_n]=\ell_p$, i.e., the unit vector bases of $\MO_{F}[s_n]$  and $\ell_p$ are equivalent.

\item[(b)] The unit vector basis of $\MO_{F}[s_n]$ is greedy.

%\item[(b)] There is a constant $K$ such that $|\{ n\in \NN \tq s_n  \ge \exp(-2^k) \}| \le \exp (K 2^k)$ for all $k\in\NN$.

\item[(c)] There is a constant $R>1$ such that 
\begin{equation}\label{targetineq}
|\{ n\in \NN \tq s_n  \ge \exp(-2^k) \}| \le R^{2^k},\quad \forall\;k\in\NN.\end{equation}

\end{enumerate}
\end{proposition}
\begin{proof}  That (a) implies (b) is obvious, and that (c) implies (a) is established in  \cite{CasazzaKalton1998}*{Proposition 6.6(3)}.
Let us show that (b) implies (c).
% Our arguments rely on those in the proof of  \cite{CasazzaKalton1998}*{Proposition 6.6(3)}. 

Assume, without lost of generality, that $s_n\le e^{-1}$ for all $n\in\NN$.
  Since the unit vector basis $\EE$ of  $\MO_{F}[s_n]$ has a subbasis equivalent to the unit vector basis of $\ell_p$ (see Remark~\ref{complementedlp})
and $\EE$ is democratic, there is a constant $0<c<1$ such that
 \begin{equation}\label{initialbound}
c N^{1/p}\leq \varphi_l[\EE, \MO_{F}[s_n]](N), \quad N\in\NN.
\end{equation}
For each $k\in\NN$ let $A_k=\{ n\in \NN \tq s_n  \ge \exp(-2^k) \}$ and put $N_k=|A_k|$. It suffices to prove \eqref{targetineq}
 for $k$ such that $cN_k\ge 1$. 
 
 By \eqref{initialbound} we have 
 \[ c N_k^{1/p}\le \left\Vert \sum_{n\in A_k} \ee_n \right\Vert_{ \MO_{F}[s_n]}. \]Hence,
\begin{align*}
1  &\le \sum_{n\in A_k} \frac{ F(c^{-1}N_k^{-1/p} s_n)}{F(s_n)}\\
&=  c^{-p}  N_k^{-1}  \sum_{n\in A_k}\left( \frac{-\log(s_n)}{-\log(s_n)+\log(cN_k^{1/p})}\right)^a\\
 &\le  c^{-p}  \left(\frac{2^k}{2^k+\log(cN_k^{1/p})}\right)^a,
\end{align*}
which yields
\[
\log(cN_k^{1/p})\le (c^{-p/a}-1) 2^k.
\]
Consequently, choosing $R=c^{-p} \exp\{p (c^{-p/a}-1)\}$ we obtain
$
N_k \le  R^{2^k}.
$
%\[
%\left\Vert \sum_{n\in A_k}\ee_n \right\Vert_{\MO_{F}[s_n]} \le N_k^{1/(p+a2^{-k})}, \quad k\in\NN. 
%\]
%Combining we obtain that, for some constant $C> 1$,
%\[
%N_k^{1/p} \le C  N_k^{1/(p+a2^{-k})}, \quad k\in\NN.
%\]
%Choosing $R=C^{(p+a)/p}$, we obtain $N_k \le R^{2^k}$ for all $k\in\NN$.
 \end{proof}

\begin{remark}\label{Identification} Suppose $F$ is a Orlicz function.  Let $(s_n)_{n=1}^\infty$ be a bounded sequence of positive numbers  which does not converge to zero. There are only two mutually exclusive possibilities:
\begin{itemize}
\item $0<\liminf_n s_n\le \limsup_n s_n<\infty$.  It is then quite obvious that $\MO_F[s_n]=\MO_F$. 

\item  $0=\liminf s_n <\limsup_n s_n<\infty$. Then $\MO_F[s_n] \approx \MO_F$  (cf. \cite{LinTza1977}*{Proposition 3.a.5}),  but this does not mean that the unit vector basis of  $\MO_F[s_n]$ is a  greedy basis. We deal with this situation in the following  corollary.

\end{itemize} 
\end{remark}
\begin{corollary}\label{uvbNotGreedy} Suppose that $p\in[1,\infty)$ and that $a\in(0,\infty)$. Consider $F=F^{p,a}$  as defined in \eqref{OurOrliczFunction}.
Let $(s_n)_{n=1}^\infty$ be a sequence such that  $0=\liminf_n s_n < \limsup_n s_n<\infty$. Then the unit vector basis of $\ell_F[s_n]$ is not a greedy basis.
\end{corollary}

\begin{proof}Suppose that the unit vector basis of $\ell_F[s_n]$ is greedy.
 There are injective mappings $\pi,\rho\colon\NN\to \NN$  such that $\lim_n  s _{\pi(n)} =0$
   and $\inf_n s_{\rho(n)} >0$.  Obviously, $\BB_1=(\xx_{\pi(n)})_{n=1}^\infty$ is isometricaly equivalent to the unit vector basis of $\ell_F[s_{\pi(n)}]$.
   Hence, by Proposition~\ref{MusielakGreedy}, $\BB_1$ is equivalent to the unit vector basis of $\ell_p$.  Analogously, 
    $\BB_2=(\xx_{\rho(n)})_{n=1}^\infty$ is isometricaly equivalent to the unit vector basis of $\ell_F[s_{\rho(n)}]$. Therefore (see Remark~\ref{Identification}) 
    $\BB_2$ is equivalent to the unit vector basis of $\wMO_F$. Consequently, denoting by $\EE$ the unit vector basis, 
    \[
    \varphi_u[\EE,\ell_p] (N) \approx \varphi_u[\BB_1,\wMO_F[s_n])(N) \approx  \varphi_u[\BB_2,\wMO_F[s_n])(N) \approx
   \varphi_u[\EE,\wMO_F](N).
   \]
   By Lemma~\ref{SameDemocracyFunction},
    $\ell_p\approx \wMO_F$. Then $t\approx t^p(-\log t)^{-a}$ at the origin,  an absurdity.
 \end{proof}

\begin{theorem}\label{MainOrlicz} Let $a>0$ and consider $F=F^{1,a}$ as defined in \eqref{OurOrliczFunction}. If $\BB$ is a complemented greedy basic sequence in $\ell_F$, then  $\BB$ is equivalent to the unit vector basis either of 
 $\ell_1$ or of $\ell_F$. Conversely, $\ell_F$ has a complemented basic sequence equivalent to the unit vector basis of $\ell_1$.
 \end{theorem}

 \begin{proof} Notice that  since $\wMO_F$ is anti-Euclidean the basic  sequence $\BB$ is lattice anti-Euclidean. By \cite{CasazzaKalton1998}*{Theorem 3.4}, $\BB$ is permutatively equivalent to a block basis of the unit vector basis of $\wMO_F$. By 
  \cite{CasazzaKalton1998}*{Lemma 6.4}, $\BB$ is  permutatively  equivalent to a normalized constant coefficient block basis.  Let $(N_n)_{n=1}^\infty$ be the sequence of the lengths of those blocks, and
choose $(s_n)_{n=1}^\infty$  such that $N_n F(s_n)=1$.  By Lemma~\ref{BlockOrliczEquivalentMusielak}, $\BB$ is equivalent to the unit vector basis of $\MO_F[s_n]$.
  
  Suppose that $\lim_{n} N_n=\infty$, hence $\lim_{n} s_n=0$.  Then, by Proposition~\ref{MusielakGreedy}, $\BB$ is equivalent to the unit vector basis of $\ell_1$.
  
   Suppose that $\sup N_n<\infty$. Then (see Remark~\ref{Identification}) $\BB$ is equivalent to the unit vector basis of $\wMO_F$.
   
   Suppose that $\liminf_n N_n <\limsup N_n=\infty$. Then, by Corollary~\ref{uvbNotGreedy}, the unit vector basis is not greedy.
     
   The converse is established in  Remark~\ref{complementedlp}.
   \end{proof}
   
   Corollaries~\ref{uniquenesslF} and \ref{ExistenceComplementedlF} are inmmediate from Theorem~\ref{MainOrlicz}.
   \begin{corollary}\label{uniquenesslF} Let $a>0$ and consider $F=F^{1,a}$  as defined  in \eqref{OurOrliczFunction}. Then $\ell_F$ has a unique greedy basis. \end{corollary}
   
    \begin{corollary}\label{ExistenceComplementedlF} Let $a>0$ and consider $F=F^{1,a}$  as defined in \eqref{OurOrliczFunction}. A complemented subspace of $\ell_F$ has a greedy basis if and only if  it is isomorphic either  to $\ell_1$ or to $\ell_F$. 
    \end{corollary}

    Notice that Corollary~\ref{uniquenesslF} solves Problem~\ref{problemon}.  
    
    We end this section providing a much more natural example than the afore-mentioned Pe{\l}czy{\'n}ski's universal space $U$,  of a space with symmetric basis having a complemented subspace without a greedy basis 
       \begin{theorem}\label{ComplementedNotGreedy} Let $a>0$ and consider $F=F^{1,a}$  as defined in \eqref{OurOrliczFunction}. There is a complemented subspace of  $\ell_F$ with an unconditional basis but with no greedy basis.     \end{theorem}
       \begin{proof} Use  Proposition~\ref{MusielakGreedy} to pick a sequence $(s_n)_{n=1}^\infty$ convergent to zero such that the unit vector basis of $\MO_F[s_n]$  is not greedy. Notice that, by Lemma~\ref{BlockOrliczEquivalentMusielak},  $\MO_F[s_n]$ is isomorphic to a complemented subspace of $\ell_F$.
       Since $\ell_1$ has a unique unconditional basis, $\MO_F[s_n]\not\approx \ell_1$.  
       
       If   $\MO_F[s_n]$  had a greedy basis,  $\MO_F[s_n]$  would be isomorphic to 
      $ \ell_F$ by Corollary~\ref{ExistenceComplementedlF}.  By \cite{CasazzaKalton1998}*{Theorem 3.5}, the unit vector basis of $\ell_F$ would be equivalent to a complemented block basis of $\MO_F^M[s_n]$  for some $M\in\NN$.
      By \cite{CasazzaKalton1998}*{Lemma 6.4}, there would be another sequence  $(s_n')_{n=1}^\infty$, decreasing to zero, such that the unit
  vector basis of $\ell_F$ is  permutatively equivalent to the unit vector basis of $\MO_F[s_n']$. In particular, the unit vector basis of $\MO_F[s_n']$ would be greedy. 
  But this is an absurdity  because,  again by Theorem~\ref{MusielakGreedy}, we would have
  $\ell_F\approx \MO_F[s_n'] \approx \ell_1$.
          \end{proof}
      
\section{A super-reflexive Banach space with a unique greedy basis}\label{superrefExam}
\noindent In this section we continue exploiting the structure of  sequence Orlicz spaces.  Consider 
$F=F^{2,1}$ as  defined  in  \eqref{OurOrliczFunction}. 
The dual of  the Orlicz sequence space $\ell_F$, under the canonical dual pair,  is the Orlicz sequence space $\ell_G$, where $G$ is a convex Orlicz function such that 
 \begin{equation}\label{dualfunction}
  G(t)= t^2(-\log(t)) \text{ for } 0<t\le e^{-3/2}
  \end{equation}
 (cf.\ \cite{LinTza1977}*{Example 4.c.1} and \cite{CasazzaKalton1998}*{Theorem 6.10}).
 Since both $F$ and $G$ satisfy the $\Delta_2$-condition at the origin it follows that $\ell_F$ 
and $\ell_G$ are super-reflexive  (see e.g. \cite[Sect.\ 1.f]{LT1979}).
  Moreover $\ell_G$ has Rademacher type $p$ for any $p< 2$ and cotype $2$ (cf.\ \cite{CasazzaKalton1996}).
 If we try to apply to $\ell_F$  the same techniques as in 
Sect.~\ref{Orlicz} we come across the obstruction that  these Orlicz spaces  are  sufficiently Euclidean (in fact $\ell_2$ is a complemented subspace of $\ell_F$, see Remark~\ref{complementedlp}).  In this new situation, it comes to our aid  \cite{CasazzaKalton1996}*{Theorem 4.3} (which  asserts that every complemented unconditional basic sequence in $\ell_G$ is equivalent to a constant coefficient block basic sequence of the unit vector basis) and the  reflexivity of  $\ell_F$ (which guarantees  that every  semi-normalized basis in   $\ell_F$  is weakly null).
 
 \begin{theorem}\label{MainSuperReflexive} Consider $F=F^{2,1}$ as in \eqref{OurOrliczFunction} and $G$ as  in \eqref{dualfunction}. If $\BB$ is a complemented greedy basic sequence in $\ell_F$ (respectively, $\ell_G$) then  $\BB$ is equivalent to the unit vector basis either of $\ell_2$ or of
 $\ell_F$  (respectively, $\ell_G$).
 \end{theorem}

 \begin{proof} Since $\ell_G^*=\ell_F$, appealing to \cite{DKKT2003}*{Theorem 5.1} it suffices to prove the result for a  complemented greedy basic sequence $\BB$ in $\ell_F$. Notice that the dual basic sequence of 
 a constant coefficient block basic sequence  is equivalent to a  constant coefficient  block basic sequence  in the dual space. Hence, 
 by  \cite{CasazzaKalton1996}*{Theorem 4.3},  $\BB$ is  permutatively  equivalent to a normalized constant coefficient block basic sequence of the unit vector basis. Now, proceed as in the proof of
  Theorem~\ref{MainOrlicz}.
   \end{proof}

Corollaries~\ref{uniquenesssSRbis} and \ref{ExistenceComplementedSRbis}  are the analogous ones to Corollaries~\ref{uniquenesslF} and \ref{ExistenceComplementedlF} in the preceding section, hence they need no further explanation.
   
   \begin{corollary}\label{uniquenesssSRbis}
Consider $F=F^{2,1}$ as  in \eqref{OurOrliczFunction} and $G$ as   in \eqref{dualfunction}. Then both spaces $\ell_F$ anf $\ell_G$ have a unique greedy basis. \end{corollary}
   
    \begin{corollary}\label{ExistenceComplementedSRbis} Consider $F=F^{2,1}$ as  in \eqref{OurOrliczFunction} and $G$ as  in \eqref{dualfunction}.  A complemented subspace of $\ell_F$ (respectively, $\ell_G$) has a greedy basis if and only if  it is isomorphic either  to $\ell_2$ or to $\ell_F$  (respectively, $\ell_G$).
    \end{corollary}
The last result of this section would be in correspondence with Theorem~\ref{ComplementedNotGreedy} in Section~\ref{Orlicz}.    
    
\begin{theorem}\label{ComplementedNotGreedySR}
Consider $F=F^{2,1}$ as in \eqref{OurOrliczFunction} and $G$ as  in \eqref{dualfunction}.  Then both spaces $\ell_F$ and $\ell_G$  have  a complemented 
subspace with an unconditional basis but with no greedy basis.    \end{theorem}
 \begin{proof}   Appealing again to \cite{DKKT2003}*{Theorem 5.1} it suffices to prove the result for $\ell_F$.
As in the proof of Theorem~\ref{ComplementedNotGreedy}, pick a sequence $(s_n)_{n=1}^\infty$ convergent to zero such that the unit vector basis of $\MO_F[s_n]$ (which is isomorphic to a complemented subspace of $\ell_F$)  is not greedy.  Since $\ell_2$ has a unique unconditional basis, $\MO_F[s_n]\not\approx \ell_2$. Suppose that $\MO_F[s_n]$  has a greedy basis $\BB$. Then, by Theorem~\ref{MainSuperReflexive}, $\BB$ must be equivalent 
  to  the unit vector basis $ \ell_F$. 
      
By \cite{CasazzaKalton1998}*{Theorem 6.6 (1)} there is a sequence of finite dimensional Banach spaces $(V_n)_{n=1}^\infty$  such that 
 $\MO_F[s_n]\approx \ell_{2}(V_n)$.  Since $\BB$ is weakly null and symmetric, appealing to the Bessaga-Pe{\l}czy{\'n}ski selection principle, we claim that it is equivalent to the unit vector basis of  $\ell_2$. This absurdity proves the result.\end{proof}
    
 \section{Uniqueness of greedy basis in Marcinkiewicz spaces}\label{Lorentz}

\noindent  Throughout this section a \textit{weight} will be a sequence of positive numbers  $\ww = (w_n)_{n=1}^\infty$  
whose \textit{primitive weight} 
$\sss=(s_n)_{n=1}^\infty$ given   by $s_n = \sum_{i=1}^n w_i $, verifies the doubling condition
$s_{2n} \lesssim s_n$ for all $n$.  In  case that  $\ww$  decreases to zero and 
$\sss$  increases to infinity, we consider the \textit{Marcinkiewicz  sequence space}  $m_\sss$,
 consisting of all  sequences $(a_n)_{n=1}^\infty$ in $c_0$  for which the following norm is finite:
$$\Vert (a_n)_{n=1}^\infty \Vert_{m_\sss} = \sup_{n \in \NN} \frac{1}{s_n} \sum_{i=1}^n a_i^*,$$
where $(a^*_n)_{n=1}^\infty$ is the decreasing rearrangement of $(|a_n|)_{n=1}^\infty$.
The \textit{separable part} of  $m_\sss$,  denoted $m^0_\sss$, is the completion of $c_{00}$ in the space $m_\sss$. 

The unit vector  basis 
 is a symmetric basis both for $m^0_\sss$ and, under the natural duality, for its dual space, the Lorentz sequence space $d_{\ww,1}$,  whose  norm
  % (norm when $\ww$ is decreasing)
is given by
$$\Vert (a_n)_{n=1}^\infty  \Vert_{\ww,1} = \sum_{n=1}^\infty a_n^* w_n.$$

We say that a weight $\ww$ is  \textit{regular} if  
\begin{equation*} %\label{eq: regularity} 
 \frac{s_n}{n} \lesssim   w_n, \quad \forall\;n\in\NN.
\end{equation*}
The regularity of  the weight  implies the following equivalence of  quasi-norms
  \begin{equation} \label{eq: weaknorms}
  \sup_{n\in\NN} \frac{a^*_n}{w_n} \lesssim \left\|\sum_{n=1}^\infty a_n \ee_n\right\|_{m_\sss} \le \sup_{n \in\NN} \frac{a^*_n}{w_n}
\end{equation}
 for all $(a_n)_{n=1}^\infty\in c_{00}$. (Note that the right hand-side inequality in \eqref{eq: weaknorms} does not require regularity.)
The inequalities in \eqref{eq: weaknorms} give us an identification between Marcinkiewicz  sequence spaces and weak Lorentz spaces. In general,
given a weight $\vv=(v_n)_{n=1}^\infty$, the space $d_{\vv,\infty}$ consists of all
sequences
$ (a_n)_{n=1}^\infty\in c_0$  such that the quasi-norm
$$\Vert (a_n)_{n=1}^\infty  \Vert_{\vv,\infty} = \sup_{n\in\NN} \left( \sum_{i=1}^n v_i \right) a_n^* $$
is finite. So, if $\ww$ is a regular weight and  $\vv=(1/w_n - 1/w_{n-1})_{n=1}^\infty$,  then $m_\sss = d_{\vv,\infty}$.
In the particular case that  for some $1<p<\infty$ the weight $\ww$ is given by
$w_n:= n^{1/p} - (n-1)^{1/p}$,  so that $\sss=(n^{1/p})_{n=1}^\infty$,
 then $\ww$ is regular and
% since $w_n\approx n^{1/p-1}$,
 the Marcinkiewicz space
$m_\sss$ coincides with the classical  weak-$\ell_q$ space  $\ell_{q,\infty}$, whose natural quasi-norm is given by
$$ \| (a_n)_{n=1}^\infty\|_{q, \infty} = \sup_{n\in\NN}  a_n^*\, n^{1/q},$$
where $q = (p-1)/p$.

We will use the following result, which illustrates the connection between Lorentz sequence spaces and greedy-like  bases. We would like to remark that it  remains valid  for  quasi-greedy bases.
\begin{lemma}\label{DemocracyEmbedding} Let $\BB=(\xx_k)_{k=1}^\infty$ be a greedy basis is a Banach space $\XX$. Let $\ww=(w_n)_{n=1}^\infty$  be a  weight and denote by $(s_n)_{n=1}^\infty$ its primitive weight.
\begin{itemize}
\item[(a)]  $\XX$ embeds in $d_{\ww,\infty}$
  via $\BB$, i.e.,
$$
\Vert (a_k)_{k=1}^\infty \Vert_{\ww,\infty}    \lesssim  \left\Vert \sum_{k=1}^\infty  a_k \xx_k\right\Vert,\quad \forall\; (a_{k})_{n=1}^\infty \in c_{00},$$ 
 if and only if
\[
s_n \lesssim \varphi_l[\BB,\XX](n), \quad \forall\,n\in\NN.\]

\item[(b)]   $d_{\ww,1}$ embeds in $\XX$ via the $\BB$, i.e.,
\[
\left\Vert \sum_{k=1}^\infty a_k \xx_n\right\Vert\lesssim \left\Vert (a_k)_{k=1}^\infty \right\Vert_{\ww,1}\,\quad\forall\;(a_k)_{k=1}^\infty\in c_{00},\]
 if and only if
\[ \varphi_u[\BB,\XX](n)  \lesssim  s_n,\quad\forall \;n\in\NN.\]
\end{itemize}
 \end{lemma}
\begin{proof} Part  (b) requires a little bit more work than part (a) but both parts  can be obtained rewriting carefully the proof of \cite{AA2015}*{Theorem2.1}. 
\end{proof}

The study of symmetric bases in Lorentz sequence spaces leads to consider basic sequences whose terms are  equidistributed disjointly supported sequences. Two sequences $(a_n)_{n=1}^\infty$ and $(b_n)_{n=1}^\infty$  of real numbers are said to be  \textit{equidistributed} if $(a_n^{\ast})_{n=1}^{\infty}=(b_n^*)_{n=1}^{\infty}$. In this direction, Altshuler el al.\   proved in \cite{ACL1973} the following theorem. 
Recall that a weight $(s_n)_{n=1}^\infty$ is called \textit{submultiplicative} if
\begin{equation*} % \label{eq: supermult}
 \sup_{n,k>0} \frac{s_{nk}}{s_n s_k} < \infty. \end{equation*}
Note that $(n^{1/p})_{n=1}^\infty$  is a submultiplicative weight for each $1<p<\infty$.
\begin{theorem}[cf. \cite{LinTza1977}*{ Theorem 4.e.5}]\label{Altshuler}Let $\ww$ be a weight  decreasing to zero 
such that its primitive weight is submultiplicative. Then any equidistributed disjointly supported basic sequence in $d_{\ww,1}$ is equivalent to the unit vector basis 
of $ d_{\ww,1}$.
\end{theorem}
 Next we  prove a lemma stating that  equidistributed disjointly supported basic sequences are {\it sufficiently far} from being equivalent to the unit vector basis of $\ell_1$, even without imposing the submultiplicative condition on the weight. 

\begin{lemma}\label{newlemma} Let $\ww$ be a weight decreasing to zero. Let $\BB$ be
%=(\yy_k)_{k=1}^\infty$ 
an  equidistributed disjointly supported basic sequence in $d_{\ww,1}$. Then there is a weight $\vv$
%=(v_n)_{n=1}^\infty$  
decreasing to zero such that $d_{\vv,1}$ embeds in $d_{\ww,1}$ via $\BB$. 
%\[
%\left\Vert \sum_{k=1}^\infty a_k \yy_k\right\Vert_{\ww,1}  \lesssim \Vert (a_n)_{n=1}^\infty \Vert_{\vv,1},
 % \sum_{n=1}^\infty  a_n^* v_n,
% \quad  (a_n)_{n=1}^\infty\in c_{00}.
%\]
\end{lemma}
\begin{proof}
%Denote $\BB=(\yy_k)_{k=1}^\infty$  and 
Let $\yy=(a_n)_{n=1}^\infty$ be a decreasing sequence of non negative numbers such that each element in $\BB$ is equidistributed with $\yy$. Then,
using Abel's summation formula, for all $N$ we have
\[
\varphi_l[\BB,d_{\ww,1}](N)=\sum_{j=1}^\infty a_j \sum_{n=(j-1)N+1}^{jN} w_n=\sum_{j=1}^\infty (a_j-a_{j+1})  \sum_{n=1}^{jN} w_n.
\]
Therefore
\[
v_N:=\varphi_l[\BB,d_{\ww,1}](N)-\varphi_l[\BB,d_{\ww,1}](N-1)=\sum_{j=1}^\infty (a_j-a_{j+1})
  \sum_{n=1+j(N-1)}^{jN} w_n.
\]
Notice that $( \sum_{n=1+j(N-1)}^{jN} w_n)_{N=1}^\infty$ decreases to zero. Then, by the dominated convergence theorem, $(v_N)_{N=1}^\infty$ decreases to zero.
The proof is over   invoking   Lemma~\ref{DemocracyEmbedding}(b).
\end{proof}

\begin{lemma} \label{lem: uncbasicseqfact} Let $\ww$ be a weight decreasing to zero such that its primitive weight $\sss$ increases to infinity.  Let $\BB$ be a semi-normalized unconditional  basic sequence in $d_{\ww,1}$. Then 
\begin{itemize} 
\item[(a)] $\BB$ has a subsequence which is equivalent either to the unit vector basis of $\ell_1$ or to a disjointly supported equidistributed sequence.
\item[(b)] If  $\sss$ is submultiplicative then $\BB$ has a subsequence which is equivalent either to the unit vector basis of $\ell_1$ or to the unit vector basis $d_{\ww,1}$.
  \end{itemize}
 \end{lemma}
\begin{proof} Taking into account Theorem~\ref{Altshuler}, we need only  prove  (a).  If $\BB=(\xx_k)_{k=1}^\infty$ is not weakly null then, by the unconditionality of $\BB$ we get  that it has a subsequence  equivalent to the unit vector basis of $\ell_1$. So we may assume that $\BB$ is weakly null.  Appealing to the Bessaga-Pe{\l}czy{\'n}ski selection principle,  we may assume that $\BB$ is a block basis of the unit vector basis. 
 We can suppose also the that there is an infinite subset $B$ of $\NN$ such that $x_{k,n}=0$ for all $k\in\NN$ and $n\in B$.
 Consider $(B_k)_{k=1}^\infty$ a partition of $\NN$ such that $\supp \xx_k \subseteq B_k$ and $B_k \setminus \supp\xx_k$ is infinite. 
 Denote,  for each $k\in\NN$, $\xx_k=(x_{k,n})_{n=1}^\infty$.
 Let $\pi_k\colon \NN \to S_k$ bijective and such that the  absolute value of $\xx_k':=(x_{k,\pi_k(j)})_{j=1}^\infty$   is a decreasing sequence. Notice that
\[
|x_{k,\pi_k(j)}| \le \frac{ \Vert \xx_k \Vert_{\ww,1} }{s_j}, \quad j,k\in\NN.
\]
Therefore, regarding  $(\xx_k')_{k=1}^\infty$ as a sequence of functions defined in the compact space $\NN\cup\{\infty\}$,  $(\xx_k')_{k=1}^\infty$ is 
 an equicontinuous and uniformly bounded sequence. By the Arzel\`a-Ascoli theorem,  passing again to a subsequence, we can suppose that there 
is $\yy=(y_j)_{j=1}^\infty \in c_0$ such that $\lim_k \xx_k'=\yy$ uniformly.  By Fatou's Lemma, $\yy\in d_{\ww,1}$.  

For each $k\in\NN$ let $\yy_k=(y_{k,n})_{n=1}^\infty$ be the sequence given by 
$y_{k,\pi_k(j)}=y_j$ for all $j\in\NN$, and $y_{k,n}=0$ if $n\notin B_k$. Let $\zz_k=\xx_k -\yy_k$.  We have:

\begin{itemize}

 \item $(\yy_k)_{k=1}^\infty$  is a disjointly supported sequence in $d_{\ww,1}$,

 \item $\lim_k \Vert  \zz_k \Vert_\infty=\lim_k \Vert  \xx_k'-\yy \Vert_\infty=0$, and

\item $\xx_k=\yy_k+\zz_k$.
\end{itemize}

\noindent If
 $\liminf_k \|\zz_k\|_{\ww,1} = 0$, then $\BB$ has a subsequence which is an arbitrarily small perturbation of a subsequence of  $(\yy_k)_{k=1}^\infty$
and hence  is equivalent to $(\yy_k)_{k=1}^\infty$. If $\liminf_k \|\zz_k\|_{\ww,1} > 0$ then, appealing to
\cite[Prop. 4.e.3]{LinTza1977}
 and passing to a subsequence, we can suppose that $(\zz_k)_{k=1}^\infty$ is equivalent to the unit vector basis of $\ell_1$. Therefore, 
\[
 A\left\Vert \sum_{k=1}^\infty a_k \ee_k\right\Vert_{1} \le \left\Vert \sum_{k=1}^\infty a_k \zz_k\right\Vert_{\ww,1}
 \]
for some constant $A$ and for all sequences  $(a_k)_{k=1}^\infty$ in $c_{00}$. 
 Furthermore,  by Lemma~\ref{newlemma}, there is a weight $\vv=(v_n)_{n=1}^\infty$ decreasing to zero and a constant $B$ such that for all   $(a_k)_{k=1}^\infty \in c_{00}$,
 \[
 \left\Vert \sum_{k=1}^\infty a_k \yy_k\right\Vert_{\ww,1}  \le B \left\Vert \sum_{k=1}^\infty a_k \ee_k\right\Vert_{\vv,1}.
\]
 Since $(\xx_k)_{k=1}^\infty$ is semi-normalized there is a constant $C$ such that for all $j\in\NN$  and  all $(a_k)_{k=1}^\infty \in c_{00}$,
\[
|a_j|  \le C \left\Vert \sum_{k=1}^\infty a_k \xx_k\right\Vert_{\ww,1}.
\]
 Combining, we obtain that $(\xx_k)_{k=1}^\infty$ is equivalent to the unit vector basis of $\ell_1$. 
 
 Indeed, 
let $N\in\NN$  be minimal with the property that $D:=A -B v_N>0$. 
Let  $(a_k)_{k=1}^\infty \in c_{00}$. Denoting by $(a_k^*)_{k=1}^\infty$ the decreasing rearrangement of 
$(|a_k|)_{k=1}^\infty$,
\begin{align*}
\left\Vert \sum_{k=1}^\infty a_k \xx_k\right\Vert_{\ww,1} 
&\ge A\sum_{k=1}^\infty a_k^*
- B \sum_{k=1}^\infty a_k^* v_k  \\
&\ge D \sum_{k=N}^\infty a_k^*+ \sum_{k=1}^{N-1} (A- B v_k)a_k^*\\
&= D\sum_{k=1}^\infty |a_k|- B \sum_{k=1}^{N-1} (v_k-v_N)a_k^*\\
& \ge D \sum_{k=1}^\infty |a_k|- BC\sum_{k=1}^{N-1}(v_k-v_N)\left\Vert \sum_{k=1}^\infty a_k \xx_k\right\Vert_{\ww,1}.
\end{align*}
This yields the desired result.
\end{proof} 

\begin{proposition} \label{prop: uncbasisfact}  Let $\ww$ be a weight decreasing to zero such that its primitive weight $\sss=(s_n)_{n=1}^\infty$ increases to infinity.
 Let $\BB$ be a semi-normalized unconditional basis of $d_{\ww,1}$
such that 
\begin{equation}\label{assumptionondemocracyf}
\varphi_u[\BB,d_{\ww,1}](n)\lesssim s_n,\quad \forall n\in \NN.\end{equation} Then 
$\BB$ is equivalent to the unit vector basis of $d_{\ww,1}$. \end{proposition} 
\begin{proof} Let $\BB=(\xx_k)_{k=1}^\infty$ and $\xx_k =(x_{k,n})_{n=1}^\infty$. Then  $$\delta:=\inf_{k\in\NN}
\| \xx_k \|_\infty >0.$$
Indeed, if not, by \cite[Prop. 4.e.3]{LinTza1977}, $\BB$ would have a subsequence equivalent to the unit vector basis of $\ell_1$, contradicting the assumption that \eqref{assumptionondemocracyf} holds.

Next, we claim that there exists $N \in \NN$ such that, for each $n\in\NN$, $|x_{k,n}| \ge \delta$ for at most $N$ values of $k$. 
 Suppose that this is not the case. Then, for every $N \ge 1$ there exist $n := n(N)$ and $k_1<k_2<\dots<k_N$ such that
$|x_{k_j,n}| \ge \delta$ for $1 \le j \le N$.  Hence there exists a choice of signs $\varepsilon_j = \pm 1$ ($1 \le j \le N$) such that
$$\left\|\sum_{j=1}^N \varepsilon_j \xx_{k_j}\right\|_{\ww,1} \ge \sum_{j=1}^N |x_{k_j,n}| \ge N\delta,$$
contradicting  again the assumption \eqref{assumptionondemocracyf} on $\varphi_u[\BB,d_{\ww,1}]$.

Hence, there is a partition $(B_j)_{j=1}^N$ of $\NN$  such that for each $k \in \NN$ there exists $n=n(k)$ such that
$|x_{k,n(k)}| \ge \delta$ and the map $k \mapsto n(k)$ is one-one on each  $B_j$.  

Now we estimate from below the norm of any element  $\sum_{n=1}^\infty a_k \xx_k \in d_{\ww,1}$.  By the unconditionality of $\BB$
(see  \cite{LT1979}*{Theorem 1.d.6}),
\[
\left \|\sum_{k=1}^\infty a_k \xx_k \right\|_{\ww,1} \approx \sum_{j=1}^N \left\|\sum_{k\in B_j}^\infty a_k \xx_k\right\|_{\ww,1}
\approx \sum_{j=1}^N  \left\| \left( \sum_{k \in B_j} |a_k|^2 |\xx_k|^2\right)^{1/2}\right\|_{\ww,1}.
\]
Then, using the estimate $|x_{k,n(k)}| \ge \delta$ and the symmetry of the unit vector basis of $d_{\ww,1}$,
\begin{align*}
\left \|\sum_{k=1}^\infty a_k \xx_k \right\|_{\ww,1}
& \gtrsim \sum_{j=1}^N
\left\| \sum_{k \in B_j} |a_k| \ee_{n(k)}\right\|_{\ww,1}	\\
&= \sum_{j=1}^N\left\| \sum_{k \in B_j} a_k \ee_k\right\|_{\ww,1}\\
& \approx \left\|\sum_{k=1}^\infty a_k \ee_k\right\|_{\ww,1}.
\end{align*}
The upper estimate follows from Lemma~\ref{DemocracyEmbedding}(b).
\end{proof}
%For the upper estimate we employ a standard Abel summation argument. For each $n\ge1$ and
%choice of coefficients $a_1,\dots,a_n$, we have
%\begin{align*} \|\sum_{i=1}^n a_i x_i\|_{w,1} &\approx \|\sum_{i=1}^n |a_i| x_i\|_{w,1} \\
%\intertext{(by unconditionality)}
%& = \|\sum_{i =1}^n a^*_i x_{\rho(i)}\|_{w,1} \\
%\intertext{(for some permutation $\rho$ of $\{1,2,\dots,n\}$)}
%&= \| \sum_{i=1}^n (a^*_i - a^*_{i+1}) (\sum_{j=1}^i x_{\rho(j)})\|_{w,1}\\
%& \le \sum_{i=1}^n (a^*_i - a^*_{i+1}) \|(\sum_{j=1}^i x_{\rho(j)})\|_{w,1}\\
%\intertext{(since, by assumption,  $\|\sum_{j=1}^i x_{\rho(j)})\|_{w,1} \lesssim s_i$)}
%& \lesssim \sum_{i=1}^n (a^*_i - a^*_{i+1}) s_i\\
%& = \|\sum_{i=1}^n a_i e_i\|_{w,1}.
%\end{align*}
%\end{proof} 
\begin{theorem} Suppose that $\ww$ is a regular weight decreasing to zero such that its primitive weight $\sss$ is submultiplicative.   Then $m^0_\sss$ has a unique greedy basis.  \end{theorem}
\begin{proof}

Let $\BB=(\xx_n)_{n=1}^\infty$ be a greedy basis for $m^0_\sss$ and let $\BB^*=(\xx_n^*)_{n=1}^\infty$ be its biorthogonal basic sequence   in  $d_{\ww,1}$. Since $m^0_\sss$ has a separable dual it follows from the unconditionality of the basis
 that  $\BB$ is shrinking, hence $\BB^*$ is a (semi-normalized and unconditional) basis of $d_{\ww,1}$.

Note that  $\BB^*$ does not contain any subsequence equivalent to the unit vector basis of  $\ell_1$ since,  otherwise, by duality, $\BB$ would contain a  subsequence equivalent to the unit vector basis of $c_0$. This in turn would imply  that
$\varphi_u[\BB,m_\sss^0]$ is bounded, which would yield that $\BB$ is equivalent to the unit vector basis of $c_0$. We would infer that $m^0_s$ is isomorphic to $c_0$, which is clearly false since the unit vector basis of  $m^0_s$ is not equivalent to the
unit vector basis of $c_0$,
and $c_0$ has a unique symmetric basis.

Therefore, by Lemma~\ref{lem: uncbasicseqfact}(b), $\BB^*$ has a subsequence equivalent to the unit vector basis of $d_{\ww,1}$. By duality
 % and unconditionality 
this implies that the corresponding subsequence of $\BB$ is equivalent to the unit vector basis of $m^0_\sss$. Using the regularity of the weight,  
\[
\varphi_l[\BB,m_\sss^0](N)\approx
\varphi_u[\BB,m_\sss^0](N)\approx \left\Vert \sum_{n=1}^N \ee_n\right\Vert = \frac{N}{s_N} \gtrsim \frac{1}{w_N},
\]
for all $N$.

Now,  \eqref{eq: weaknorms} and Lemma~\ref{DemocracyEmbedding}(a), allow us  to estimate from below the norm of any element  
 $\sum_{k=1}^\infty a_k \xx_k \in m^0_s$. We have
\[
 \left \|\sum_{k=1}^\infty a_k \xx_k \right \|_{m_\sss} 
\gtrsim  \sup_{k\in\NN}  \frac{1}{w_k} a^*_k
\ge  
\left\Vert   (a_k)_{k=1}^\infty\right\Vert_{m_\sss}.
\]
By duality, for all $(a_n)_{n=1}^\infty \in c_{00}$ we have
$$\left\|\sum_{k=1}^\infty a_k \xx^*_k\right\|_{\ww,1} \lesssim \|\ (a_k)_{k=1}^\infty   \|_{\ww,1}.$$ 
Appealing to  Lemma~\ref{DemocracyEmbedding}(b) we obtain that 
\[
\varphi_u[\BB^*, d_{\ww,1}] (n)\lesssim s_n,\quad \forall n,\] which implies 
by Proposition~\ref{prop: uncbasisfact} that $\BB^* $ is equivalent to the unit vector basis of $d_{\ww,1}$. Hence, by duality, $\BB$ is equivalent to the unit vector basis of $m^0_\sss$.
\end{proof} \begin{corollary} The separable part of weak-$\ell_p$ has a unique greedy basis. \end{corollary}

\subsection*{Acknowledgements} F. Albiac and J. L. Ansorena acknowledge the support of Spanish Research Grant MTM2014-53009-P, {\it An\'alisis Vectorial, Multilineal, y Aplicaciones}. The first-named author  was also partially supported by Spanish Research Grant MTM2012-31286, {\it Operators, lattices, and structure of Banach spaces}. S. J. Dilworth was partially supported by the National Science Foundation under Grant Number DMS-1361461. D. Kutzarova has been partially supported by the Bulgarian National
Scientific Fund under Grant DFNI-I02/10.

% ------------------------------------------------------------------------

% ------------------------------------------------------------------------
\end{document}